\documentclass[11pt]{amsart}

\usepackage{amsmath, amssymb,amsthm}
\usepackage{color}
\usepackage{graphicx}
\usepackage{colortbl}
\usepackage{mathtools}
\usepackage{float}
\usepackage{bm}
\usepackage{biblatex}
\usepackage{caption}
\usepackage{tikz}
\newcommand{\Plus}{\mathord{\begin{tikzpicture}[baseline=0ex, line width=1, scale=0.13]
\draw (1,0) -- (1,2);
\draw (0,1) -- (2,1);
\end{tikzpicture}}}

\newtheorem{theorem}{Theorem}[section]

\newtheorem{lemma}[theorem]{Lemma}

\newtheorem{proposition}[theorem]{Proposition}

\newlength\cellsize 

\savebox2{%
\begin{picture}(15,15)
\put(0,0){\line(1,0){15}}
\put(0,0){\line(0,1){15}}
\put(15,0){\line(0,1){15}}
\put(0,15){\line(1,0){15}}
\end{picture}}

\savebox3{%
\begin{picture}(15,15)
\put(0,0){\line(1,0){17}}
\put(0,15){\line(1,0){17}}
\end{picture}}

\newcommand\cellify[1]{\def\thearg{#1}\def\nothing{}%
\ifx\thearg\nothing\vrule width0pt height\cellsize depth0pt%
  \else\hbox to 0pt{\usebox2\hss}\fi%
  \vbox to 15\unitlength{\vss\hbox to 15\unitlength{\hss$#1$\hss}\vss}}

\newcommand\tableau[1]{\vtop{\let\\=\cr
\setlength\baselineskip{-12000pt}
\setlength\lineskiplimit{12000pt}
\setlength\lineskip{0pt}
\halign{&\cellify{##}\cr#1\crcr}}}

\newcommand{\graybox}{\textcolor[RGB]{165,165,165}{\rule{1\cellsize}{1\cellsize}}\hspace{-\cellsize}\usebox2}

\title{A cornucopia of quasi-Yamanouchi tableaux}  

\author[George Wang]{George Wang}
\address{Department of Mathematics, University of Pennsylvania, Philadelphia, PA, 19104}
\email{wage@math.upenn.edu}

\subjclass[2010]{Primary 05A19; Secondary 05A15, 05E05}

\date{\today}

\keywords{tableau, descent, statistic, enumeration}

\begin{document}

\maketitle

\begin{abstract}
    Quasi-Yamanouchi tableaux are a subset of semistandard Young tableaux and refine standard Young tableaux. They are closely tied to the descent set of standard Young tableaux and were introduced by Assaf and Searles to tighten Gessel's fundamental quasisymmetric expansion of Schur functions. The descent set and descent statistic of standard Young tableaux repeatedly prove themselves useful to consider, and as a result, quasi-Yamanouchi tableaux make appearances in many ways outside of their original purpose.
 Some examples, which we present in this paper, include the Schur expansion of Jack polynomials, the decomposition of Foulkes characters, and the bigraded Frobenius image of the coinvariant algebra. While it would be nice to have a product formula enumeration of quasi-Yamanouchi tableaux in the way that semistandard and standard Young tableaux do, it has previously been shown by the author that there is little hope on that front. The goal of this paper is to address a handful of the numerous alternative enumerative approaches. 
In particular, we present enumerations of quasi-Yamanouchi tableaux using $q$-hit numbers, semistandard Young tableaux, weighted lattice paths, and symmetric polynomials, as well as the fundamental quasisymmetric and monomial quasisymmetric expansions of their Schur generating function.
\end{abstract}

\tableofcontents

\section{Introduction}

Schur polynomials have an elegant combinatorial description as a sum of monomials indexed by semistandard Young tableaux. However, as the number of variables increases, the number of semistandard Young tableaux rises dramatically, quickly making a computation using this description intractable. Ira Gessel \cite{Ge84} defined the fundamental basis for quasisymmetric polynomials and proved that Schur polynomials can instead be expressed as a sum of fundamental quasisymmetric polynomials indexed by standard Young tableaux. Since the number of standard Young tableaux of a given shape is fixed, this provides a significant computational improvement over the unbounded number of semistandard Young tableaux. Still, when the number of variables is low enough, it turns out that some of the standard Young tableaux contribute nothing to the expansion, indicating that some further improvement can be made. Sami Assaf and Dominic Searles \cite{AsSe17} introduced the concept of quasi-Yamanouchi tableaux and proved that we can tighten Gessel's expansion by summing over these tableaux instead, which give exactly the nonzero terms.

Quasi-Yamanouchi tableaux are in bijection with standard Young tableaux of the same shape, and through that bijection, quasi-Yamanouchi tableaux refine standard Young tableaux by their descent statistic.
In general, descent statistics have formed a rich and fruitful area, for example in the past with Solomon's descent algebra \cite{So76} and subsequent study of it \cite{GR89} or, more recently, work on shuffle-compatible permutation statistics by Gessel and Zhuang \cite{GZ18}. Descents of standard Young tableaux do not appear to be an exception.
Quasi-Yamanouchi tableaux inherit this value, and consequently, applications are readily found. In \textsection7, we discuss some of these applications, which involve the Schur expansion of Jack polynomials, the decomposition of Foulkes characters into irreducible characters, and the bigraded Frobenius image of the coinvariant algebra. 

Semistandard Young tableaux and standard Young tableaux are each enumerated by celebrated product formulas, and one would hope that the same could be said for quasi-Yamanouchi tableaux. Unfortunately, in \cite{Wa16}, the author demonstrates that there is likely only a product formula for certain special cases due to large primes appearing for general shapes. 
The work in this paper takes a number of alternative approaches and was inspired by the surprising appearance of quasi-Yamanouchi tableaux in the coefficients of Jack polynomials under a binomial coefficient basis \cite{FPSAC} and a hit number interpretation of the same coefficients. Comparing the two equivalent interpretations of the coefficients gives an enumeration of quasi-Yamanouchi tableaux in terms of hit numbers of certain Ferrers boards. We reproduce this enumeration in \textsection3 and prove two $q$-analogues involving the major index and charge statistics. In \textsection4, we give a summation formula in terms of semistandard Young tableaux, and in \textsection5 we first give an enumeration in terms of weighted lattice paths and then using symmetric polynomials. Finally, we consider the Schur generating function in \textsection6 and give its fundamental quasisymmetric expansion and monomial symmetric expansion.

\section{Preliminaries}

We first present the more general concepts that will be used throughout the paper; the more specific concepts will be located at the start of their relevant section. A \emph{partition} $\lambda = (\lambda_1 \geq \ldots \geq \lambda_k > 0)$ is a weakly decreasing sequence of positive integers. The \emph{size} of $\lambda$ is denoted $|\lambda|$ and is the sum of the integers of the sequence. The \emph{length} of $\lambda$, $\ell(\lambda)$, is the number of integers in the partition. Let $n(\lambda) = \sum_{i=1}^k (i-1)\lambda_i$. We say that a partition $\lambda$ dominates a partition $\mu$ if for all $i \geq 1$, $\lambda_1+\cdots+\lambda_i \geq \mu_1+\cdots+\mu_i$. 

We identify a partition with its \emph{diagram}, where rows are counted from bottom to top, the number of boxes in the $i$th row equals $\lambda_i$, and boxes are left justified. The conjugate $\lambda$ is written $\lambda'$ and is obtained by reflecting the diagram across the diagonal. We identify $u=(i,j)$ with the box in the $i$th column and $j$th row. The \emph{content} of a square is $c(u)=i-j$, and the \emph{hook-length} of a box is $h(u)$, which counts the number of boxes $(a,b)$ such that either $a>i$ or $b>j$ or $a=i,b=j$.

\begin{figure}[ht]
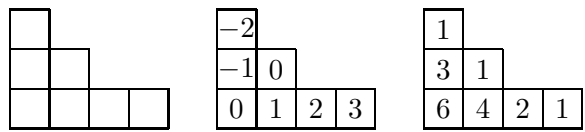

    \centering
   $$\tableau{\ \\ \ & \ \\ \ & \ & \ & \ \\}
   \ \ \ \ \ 
   \tableau{-2 \\ -1 & 0 \\ 0 & 1 & 2 & 3 \\}
   \ \ \ \ \ 
   \tableau{1 \\ 3 & 1 \\ 6 & 4 & 2 & 1 \\}$$
    \caption{The diagram of $(4,2,1)$, the cell contents, and the cell hook-lengths.}
    \label{fig:my_label}
\end{figure}

We write permutations $\pi \in S_n$ in one line notation, $\pi = \pi_1\cdots\pi_n$, where $\pi_i = \pi(i)$. The \emph{descent set} of $\pi$ is $\textnormal{Des}(\pi) = \{ i \in [n-1]\ | \ \pi_i > \pi_{i+1}\}$, and we denote the size $|Des(\pi)|$ by $des(\pi)$. The major index for a permutation is $\textnormal{maj}(\pi) = \sum_{i\in \textnormal{Des}(\pi)} i$. 

Bases of the ring of symmetric functions are indexed by partitions, and in particular, $e_\lambda$, $m_\lambda$, and $s_\lambda$ are the elementary symmetric, monomial symmetric, and Schur functions respectively.
We also write $F_\sigma(x)$ for the fundamental quasisymmetric function, where $\sigma \subseteq \{1,\ldots, n-1\}$ and 
$F_\sigma(x) = \sum_{\substack{i_1 \leq \cdots \leq i_n \\ j\in \sigma \Rightarrow i_j < i_{j+1}}} x_{i_1}\cdots x_{i_n}$.

A \emph{semistandard Young tableau} (SSYT) is a filling of a partition $\lambda$ using positive integers that weakly increase to the right and strictly increase upwards, and the set of such fillings is denoted $\textnormal{SSYT}(\lambda)$. When the entries are unbounded, there are infinitely many semistandard Young tableaux of a given shape, so it is sometimes useful to consider instead $\textnormal{SSYT}_m (\lambda)$, the set of SSYT of shape $\lambda$ with entries at most $m$. We will engage in some typical abuse of notation by using these to also mean the number of semistandard Young tableaux in the set, and we repeat this abuse later when we define standard Young tableaux and quasi-Yamanouchi tableaux. $\textnormal{SSYT}_m (\lambda)$ is enumerated by Stanley's hook-content formula, which we reproduce below. 

\begin{theorem}[Hook-content formula \cite{Sta71}] Given a partition $\lambda$, 
$$\textnormal{SSYT}_m (\lambda) = \prod_{u\in \lambda}\frac{m+c(u)}{h(u)}.$$
\end{theorem}

\begin{figure}[ht]
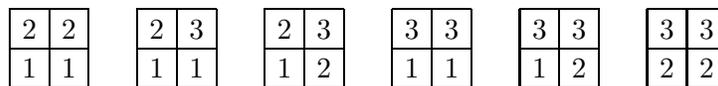

    \centering
    $$\tableau{2&2\\ 1&1\\}
\ \ \ \ \ 
\tableau{2&3\\ 1&1\\}
\ \ \ \ \ 
\tableau{2&3\\ 1&2\\}
\ \ \ \ \ 
\tableau{3&3\\ 1&1\\}
\ \ \ \ \ 
\tableau{3&3\\ 1&2\\}
\ \ \ \ \ 
\tableau{3&3\\ 2&2\\}
$$
    \caption{All 6 elements of SSYT$_3(2,2)$.}
    \label{figSSYT}
\end{figure}

The weight of an SSYT $T$ is $wt(T) = (t_1,t_2,\ldots)$, where $t_i$ is the number of times that $i$ appears, and given partitions $\lambda,\mu$, the \emph{Kostka} numbers $K_{\lambda\mu}$ count the number of SSYT of shape $\lambda$ and weight $\mu$. 
A \emph{standard Young tableau} (SYT) of shape $\lambda$ with size $n$ is a semistandard Young tableau of shape $\lambda$ with weight $(1^n)$, and the set of such fillings is denoted $\textnormal{SYT}(\lambda)$. Frame, Robinson, and Thrall counted standard fillings using the hook-length formula.

\begin{theorem}[Hook-length formula \cite{FRT54}] Given a partition $\lambda$,
$$\textnormal{SYT}(\lambda) = \frac{n!}{\prod_{u\in \lambda} h(u)}.$$
\end{theorem}

\begin{figure}[ht]
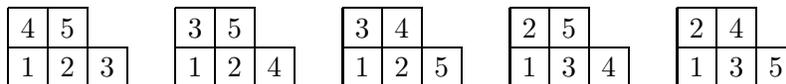

    \centering
    $$\tableau{4&5\\1&2&3}
\ \ \ \ \ 
\tableau{3&5\\1&2&4}
\ \ \ \ \ 
\tableau{3&4\\1&2&5}
\ \ \ \ \ 
\tableau{2&5\\1&3&4}
\ \ \ \ \ 
\tableau{2&4\\1&3&5}
$$
    \caption{All 5 elements of SYT$(3,2)$.}
    \label{figSYT}
\end{figure}

The descent set for $T\in \textnormal{SYT}(\lambda)$ is $\textnormal{Des}(T) = \{ i \in [n-1] \ | \ i+1 \textnormal{ is above } i\}$.
If we write the descent set as $\{d_1,d_2,\ldots,d_{k-1}\}$ in increasing order, then the first \emph{run} of the tableau is the set of boxes that contain all the entries from $1$ to $d_1$. For $1<i<k$, the $i$th run is the set of boxes containing entries from $d_{i-1}+1$ to $d_i$, and the $k$th run starts at $d_k+1$ and ends at $n$. 

\begin{figure}[ht]
    \centering
    $$\tableau{\mathbf{9} &\mathbf{10} &\textcolor[RGB]{140,140,140}{12} \\
\textcolor[RGB]{140,140,140}4 &\textcolor[RGB]{140,140,140}5 &\textcolor[RGB]{140,140,140}7 &\mathbf{11} \\
\textcolor[RGB]{140,140,140}{1} &\textcolor[RGB]{140,140,140}{2} &\textcolor[RGB]{140,140,140}{3} &\textcolor[RGB]{140,140,140}6 &\textcolor[RGB]{140,140,140}8 \\ }
$$
    \caption{This tableau has descent set $\{3,6,8,11\}$ and has the fourth run bolded.}
    \label{figRuns}
\end{figure}

An SSYT is a \emph{quasi-Yamanouchi tableau} (QYT) if when $i$ appears in the tableau, some instance of $i$ is in a higher row than some instance of $i-1$ for all $i$. We write $\textnormal{QYT}(\lambda)$ to denote the set of QYT of shape $\lambda$, $\textnormal{QYT}_{\leq m}(\lambda)$ to denote those with largest entry at most $m$, and $\textnormal{QYT}_{=m}(\lambda)$ to denote those with largest entry exactly $m$. As mentioned before, at times we will also write these to mean the number of objects in the set. 

\begin{figure}[ht]
    \centering
    $$
\tableau{4\\
2&3\\
1&2&2&4\\
}
\ \ \ \ \ 
\tableau{
4\\
3&3\\
1&2&2&5\\}
$$
    \caption{The left is a quasi-Yamanouchi filling, while the right is not.}
    \label{figQYTex}
\end{figure}

\begin{figure}[ht]
    \centering
    $$\tableau{3\\
2&2\\
1&1\\}
\ \ \ \ \ 
\tableau{3\\
2&3\\
1&1\\}
\ \ \ \ \ 
\tableau{3\\
2&3\\
1&2\\}
\ \ \ \ \ 
\tableau{4\\
2&3\\
1&2\\}
\ \ \ \ \ 
\tableau{3\\
2&4\\
1&3\\}
$$
    \caption{QYT of shape $(2,2,1)$, showing that $\textrm{QYT}_{=3}(2,2,1) = 3$ and $\textrm{QYT}_{=4}(2,2,1) = 2$.}
    \label{figQYTcount}
\end{figure}

It turns out that there is a nice bijection between $\textnormal{QYT}(\lambda)$ and $\textnormal{SYT}(\lambda)$ via the following destandardization map \cite{AsSe17}. Given a tableau $T\in \textnormal{SYT}(\lambda)$, its destandardization is the SSYT constructed by changing the value of all entries in the $i$th run of $T$ to $i$. It is easy to see that the resulting SSYT is also quasi-Yamanouchi and that this map has a well defined inverse, which we call standardization. By construction, it is clear that the map sents SYT with $k$ runs to QYT with maximum value $k-1$. In this sense, the set $\{\textnormal{QYT}_{=m}(\lambda)\}$ refines $\textnormal{SYT}(\lambda)$ by number of descents, and we will often identify a QYT with its standardization. For further discussion on properties of QYT, see \cite{Wa16}.

As with permutations, we define the \emph{major index} of a tableau $T\in \textnormal{SYT}(\lambda)$ to be $\textnormal{maj}(T) = \sum_{i\in \textnormal{Des}(T)} i$. We also have the \emph{charge} statistic for standard Young tableaux: each entry $i$ in $T$ has a charge defined recursively, where $ch(1) = 0$, $ch(i+1) = ch(i)$ if $i \not\in \textnormal{Des}(T)$, $ch(i+1) = ch(i) +1$ if $i \in \textnormal{Des}(T)$, and $ch(T) = \sum_{i=1}^{|\lambda|} ch(i)$. 
For a quasi-Yamanouchi tableau, we define its descent set, major index, and charge statistic to be those of its standardization.

\section{Hit number formulas}

Fix $n \in \mathbb{N}$. A \emph{board} $B$ is a subset of $[n]\times [n]$, which we view as an $n$ by $n$ array of squares. In particular, a \emph{Ferrers board} is a board composed of adjacent columns whose heights are weakly increasing from left to right. If $\lambda$ is a partition of size $n$, then we can construct a Ferrers board $B_\lambda$ as follows. Take the contents $c_1,\ldots,c_n$ of $\lambda$ arranged in weakly decreasing order, then let the heights of the columns of $B_\lambda$ be $(c_i+i-1)$. Let $B_\lambda \times 1$ be the board obtained by incrementing the height of every column by one. We note that it is always possible to do this once for any $B_\lambda$, as the construction never creates a column with height $n$. 
From the definitions, we get a relation between $B_\lambda$ and $B_{\lambda'}$. 
\begin{proposition}\label{Blambda}
Given a partition $\lambda$, the complement of $B_\lambda \times 1$ is $B_{\lambda'}$ up to rotation.
\end{proposition}

\begin{figure}[ht]
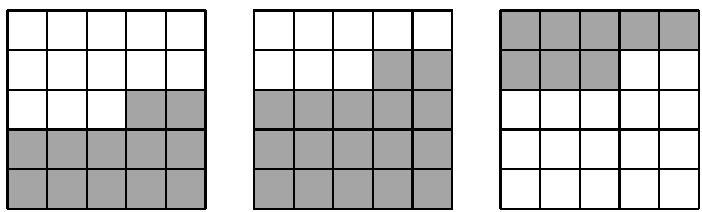

    \centering
    $$\tableau{
    \ & \ & \ & \ & \ \\
    \ & \ & \ & \ & \ \\
    \ & \ & \ & \graybox & \graybox \\
    \graybox & \graybox & \graybox & \graybox & \graybox \\
    \graybox & \graybox & \graybox & \graybox & \graybox \\}
    \ \ \ \ \ 
    \tableau{
    \ & \ & \ & \ & \ \\
    \ & \ & \ & \graybox & \graybox \\
    \graybox & \graybox & \graybox & \graybox & \graybox \\
    \graybox & \graybox & \graybox & \graybox & \graybox \\
    \graybox & \graybox & \graybox & \graybox & \graybox \\}
    \ \ \ \ \ 
    \tableau{
    \graybox & \graybox & \graybox & \graybox & \graybox \\
    \graybox & \graybox & \graybox & \ & \ \\
    \ & \ & \ & \ & \ \\
    \ & \ & \ & \ & \ \\
    \ & \ & \ & \ & \ \\}$$
    \caption{$B_{(3,2)}$, $B_{(3,2)}\times 1$, and $B_{(2,2,1)}$ rotated.}
    \label{figBlambda}
\end{figure}

Given a permutation $\pi \in S_n$, the graph of $\pi$ is $\Gamma(\pi) = \{ (i,\pi_i) \ | \ 1 \leq i \leq n\}$, and the number of \emph{hits} of $\pi$ on $B$ is $|B\cap \Gamma(\pi)|$. We define the $k$th hit number $h_k(B)$ to be the number of permutations in $S_n$ which have exactly $k$ hits on $B$.
Dworkin \cite{Dw98} gave a combinatorial interpretation of a $q$-analogue of hit numbers for Ferrers boards, which we use as the definition. For $\pi \in S_n$, place a cross at each square in $\Gamma(\pi)$, and for any square to the right of a cross, put a bullet. Then from each cross, draw circles going up and wrapping around the top edge of the $[n]\times[n]$ array, skipping over bullets, and stopping after hitting the top border of the Ferrers board. The $q$ weight of $\pi$ is the number of circles at the end of this process. The $k$th $q$-hit number $T_k(B)$ is the sum of $q$ weights over all permutations that hit the board exactly $k$ times.

\begin{figure}[ht]
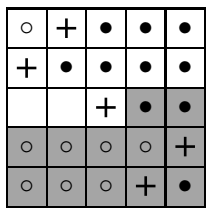

    \centering
$$\tableau{
    \circ & \Plus & \bullet & \bullet & \bullet \\
    \Plus & \bullet & \bullet & \bullet & \bullet \\
    \ & \ & \Plus & \mathclap{\graybox} \mathclap{\raisebox{5pt}{$\bullet$}} & \mathclap{\graybox} \mathclap{\raisebox{5pt}{$\bullet$}} \\
    \mathclap{\graybox} \mathclap{\raisebox{5pt}{$\circ$}} & \mathclap{\graybox} \mathclap{\raisebox{5pt}{$\circ$}} & \mathclap{\graybox} \mathclap{\raisebox{5pt}{$\circ$}} & \mathclap{\graybox} \mathclap{\raisebox{5pt}{$\circ$}} &\mathclap{\graybox} \mathclap{\raisebox{4pt}{$\Plus$}} \\
    \mathclap{\graybox} \mathclap{\raisebox{5pt}{$\circ$}} & \mathclap{\graybox} \mathclap{\raisebox{5pt}{$\circ$}} & \mathclap{\graybox} \mathclap{\raisebox{5pt}{$\circ$}} & \mathclap{\graybox} \mathclap{\raisebox{4pt}{$\Plus$}} & \mathclap{\graybox} \mathclap{\raisebox{5pt}{$\bullet$}} \\}$$
    \caption{The $q$ weight of $45312$ on $B_{3,2}$ is $8$.}
    \label{figqHit}
\end{figure}

In \cite{FPSAC}, two separate interpretations of the coefficients of the Schur expansion of the one row case of Jack polynomials are given. One is in terms of quasi-Yamanouchi tableaux, and the other is in terms of hit numbers. By comparing these interpretations, we get the following hit number formula for quasi-Yamanouchi tableaux.
\begin{theorem}\label{qytHit}
Given a partition $\lambda$ of $n$ and $0\leq k \leq n-1$,
$$\textnormal{QYT}_{=k+1}(\lambda) = \frac{h_k(B_{\lambda'})}{\prod_{u\in \lambda} h(u)}.$$
\end{theorem}
In this section, we prove two $q$-analogues of this theorem, where the first weights each tableaux by its major index and the second by its charge. To do this, we use the theory of posets and $(P,\omega)$-partitions, which were introduced by Stanley \cite{OSAP}.

For a partition $\lambda = (\lambda_1,\ldots,\lambda_k)$, let $P_\lambda$ be the subposet of $\mathbb{N}\times\mathbb{N}$ such that $(i,j)\in P_\lambda$ if $1\leq j \leq k$, $1\leq i\leq \lambda_j$. Given a poset $P$ with $n$ elements, a \emph{labeling} $\omega$ is a map $\omega :P \to [n]$. It is called a \emph{natural} labeling if it is order preserving and \emph{strict} if it is order reversing. For $P_\lambda$, there are also \emph{column-strict} labelings, which are strict on columns and natural on rows.

\begin{figure}[ht]
    \centering
    \begin{tikzpicture}[line width = 1, scale=0.65]
    \filldraw[black] (0,3) circle (2pt) node[anchor=west]{};
    \filldraw[black] (1,2) circle (2pt) node[anchor=west]{};
    \filldraw[black] (2,1) circle (2pt) node[anchor=west]{};
    \filldraw[black] (3,0) circle (2pt) node[anchor=west]{};
    \filldraw[black] (1,4) circle (2pt) node[anchor=west]{};
    \filldraw[black] (2,3) circle (2pt) node[anchor=west]{};
    \filldraw[black] (3,2) circle (2pt) node[anchor=west]{};
    \filldraw[black] (4,1) circle (2pt) node[anchor=west]{};
    \filldraw[black] (5,2) circle (2pt) node[anchor=west]{};
\filldraw[black] (4,3) circle (2pt) node[anchor=west]{};
\filldraw[black] (6,3) circle (2pt) node[anchor=west]{};
\draw[black] (3,0) -- (4,1);
\draw[black] (5,2) -- (4,1);
\draw[black] (3,0) -- (6,3);
\draw[black] (3,0) -- (2,1);
\draw[black] (1,2) -- (2,1);
\draw[black] (1,2) -- (0,3);
\draw[black] (4,1) -- (1,4);
\draw[black] (2,1) -- (4,3);
\draw[black] (5,2) -- (4,3);
\draw[black] (1,2) -- (2,3);
\draw[black] (0,3) -- (1,4);
\end{tikzpicture}
\ \ \ \ \ \ \ 
\begin{tikzpicture}[line width = 1, scale=0.65]
\filldraw[black] (0,3) circle (2pt) node[anchor=west]{1};
\filldraw[black] (1,2) circle (2pt) node[anchor=west]{3};
\filldraw[black] (2,1) circle (2pt) node[anchor=west]{5};
\filldraw[black] (3,0) circle (2pt) node[anchor=west]{8};
\filldraw[black] (1,4) circle (2pt) node[anchor=west]{2};
\filldraw[black] (2,3) circle (2pt) node[anchor=west]{4};
\filldraw[black] (3,2) circle (2pt) node[anchor=west]{6};
\filldraw[black] (4,1) circle (2pt) node[anchor=west]{9};
\filldraw[black] (5,2) circle (2pt) node[anchor=west]{10};
\filldraw[black] (4,3) circle (2pt) node[anchor=west]{7};
\filldraw[black] (6,3) circle (2pt) node[anchor=west]{11};
\draw[black] (3,0) -- (4,1);
\draw[black] (5,2) -- (4,1);
\draw[black] (3,0) -- (6,3);
\draw[black] (3,0) -- (2,1);
\draw[black] (1,2) -- (2,1);
\draw[black] (1,2) -- (0,3);
\draw[black] (4,1) -- (1,4);
\draw[black] (2,1) -- (4,3);
\draw[black] (5,2) -- (4,3);
\draw[black] (1,2) -- (2,3);
\draw[black] (0,3) -- (1,4);
\end{tikzpicture}
    \caption{$P_{4,3,2,2}$ and a column-strict labeling of $P_{4,3,2,2}$.}
    \label{fig:my_label}
\end{figure}
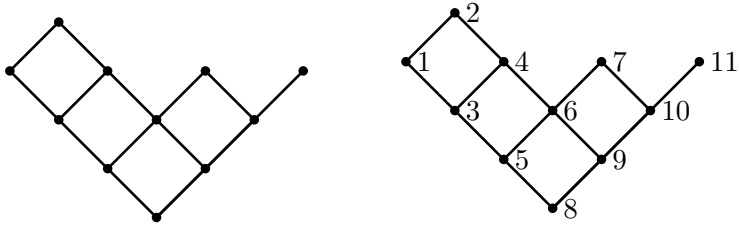

\noindent For a fixed $\omega$, a $(P,\omega)$\emph{-partition} of size $p$ is a map $\sigma : P \to \mathbb{N}_{\geq 0}$ satisfying 

1) $x\leq y \in P \implies \sigma(x)\geq \sigma(y)$, meaning $\sigma$ is order reversing.

2) $x < y \in P$ and $\omega(x)>\omega(y) \implies \sigma(x) > \sigma(y)$.

3) $|\sigma| = \sum_{x\in P} \sigma(x) = p$.

\noindent The values $\sigma(x)$ are called the parts of $\sigma$, and a $(P,\omega;m)$-partition is a $(P,\omega)$-partition with largest part at most $m$. $\mathcal{A}(P,\omega)$ denotes the set of $(P,\omega)$-partitions, and $\mathcal{A}(P,\omega;m)$ denotes the set of $(P,\omega;m)$-partitions, which have generating function
$$U_m(P,\omega;m) = \sum_{\sigma\in \mathcal{A}(P,\omega;m)} q^{|\sigma|}.$$
The $\omega$\emph{-separator} $\mathcal{L}(P,\omega)$ is the set of permutations in $S_n$ of the form $\omega(x_{i_1}),\ldots,\omega(x_{i_n})$ where $x_{i_1}<\ldots<x_{i_n}$ forms a linear extension of $P$. For each $0\leq k\leq n-1$, define
$$W_k(P,\omega) = W_k(P,\omega;q) = \sum_{\pi} q^{\textnormal{maj}(\pi)},$$
where the sum is over all $\pi \in \mathcal{L}(P,\omega)$ with $des(\pi) = k$. 

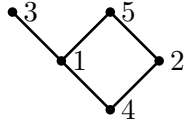
\begin{figure}[ht]
    \centering
    \begin{tikzpicture}[line width = 1, scale = 0.65]
    \filldraw[black] (0,2) circle (2pt) node[anchor=west]{3}; 
    \filldraw[black] (1,1) circle (2pt) node[anchor=west]{1}; 
    \filldraw[black] (2,0) circle (2pt) node[anchor=west]{4};
    \filldraw[black] (2,2) circle (2pt) node[anchor=west]{5}; 
    \filldraw[black] (3,1) circle (2pt) node[anchor=west]{2}; 
    \draw[black] (0,2) -- (2,0);
    \draw[black] (1,1) -- (2,2);
    \draw[black] (2,0) -- (3,1);
    \draw[black] (2,2) -- (3,1);
    \end{tikzpicture}
    \caption{$\mathcal{L}(P,\omega) = \{42153,42135,41235,41253,41325\}$ \\and $W_2(P,\omega;q) = q^3 + q^4 + q^5$.}
    \label{fig:my_label}
\end{figure}

\subsection{Major index formula}
Fix a partition $\lambda$, and let $\omega$ be a column-strict labeling on $P_\lambda$. By Proposition 21.3 of \cite{OSAP},
$$U_m(P_\lambda,\omega;m) = q^{n(\lambda)} \prod_{u\in\lambda} \frac{[m+c(u)+1]}{[h(u)]}.$$
By the definition of $W_k(P_\lambda,\omega)$, when $\omega$ is column-strict,
$$W_k(P_\lambda,\omega) = \sum_{T\in \textrm{QYT}_{=k+1}(\lambda)} q^{\textrm{maj}(T)}.$$
Proposition 8.2 of \cite{OSAP} gives
$$U_m(P_\lambda,\omega) = \sum_{k=0}^{n-1}{m+n-k\brack n} W_k(P_\lambda,\omega).$$
Then since there is no restriction on $m$, it follows that
$$\sum_{k=0}^{n-1}{x+n-k\brack n} W_k(P_\lambda,\omega) = q^{n(\lambda)}\prod_{u\in\lambda} \frac{[x+c(u)+1]}{[h(u)]}.$$
We then apply to the right hand side the following $q$-analogue \cite{Ha98} of the Goldman, Joichi, Write identity \cite{GJW75} 
$$\prod_{i=1}^n [x+b_i -i +1] = \sum_{k=0}^n {x+k\brack n} T_k(B),$$
where $B$ is a Ferrers board with column heights $b_i$. Comparing coefficients of ${x+k\brack n}$ gives the following theorem.
\begin{theorem}\label{qytMaj}
Given a partition $\lambda$ and $0\leq k \leq n-1$,
$$\sum_{T\in \textrm{QYT}_{=k+1}(\lambda)} q^{\textrm{maj}(T)} = \frac{q^{n(\lambda)}}{\prod_{u\in\lambda}[h(u)]}T_{n-k}(B_\lambda \times 1).$$
\end{theorem}
Setting $q=1$ and applying Proposition \ref{Blambda} recovers Theorem \ref{qytHit}.
We note that since $T_k(B)$ is Mahonian \cite{Dw98} for a Ferrers board, summing over $k$ gives a nice (known) $q$-analogue of the hook-length formula,
$$\sum_{T\in\textrm{SYT}(\lambda)} q^{\textrm{maj}(T)} = \frac{q^{n(\lambda)}[n]!}{\prod_{u\in\lambda} [h(u)]}.$$
We briefly attempted to prove Theorem \ref{qytMaj} bijectively but were unsuccessful. It would be nice to know what such a bijective algorithm might look like, and such an algorithm could be an interesting project to revisit in the future.

\subsection{Charge formula}
Fix a permutation $\lambda$ of size $n$, and let $\omega$ be a column-strict labeling on $P_\lambda$. We write $P_\lambda^*$ for the dual of $P_\lambda$ and write $\omega^*$ for the labeling defined by $\omega^*(x_i) = n+1-\omega(x_i)$ for all $x_i\in P_\lambda$. 
Proposition 12.1 of \cite{OSAP} details what this dualization on $P_\lambda$ and $\omega$ does to $W_k$, which is that
$$W_k(P_\lambda^*,\omega^*;q) = q^{nk} W_k(P_\lambda,\omega;\frac{1}{q}).$$
We note that since there are $k$ descents,
$$q^{nk}W_k(P_\lambda,\omega;\frac{1}{q}) = \sum_{T\in \textrm{QYT}_{=k+1}(\lambda)}q^{nk-\textrm{maj}(T)} = \sum_{T\in \textrm{QYT}_{=k+1}(\lambda)} \sum_{i\in\textrm{Des}(T)} q^{n-i}.$$
Then since a descent at position $i$ increments the charge value of the $n-i$ remaining entries by one, we get
$$W_k(P_\lambda^*,\omega^*;q) = \sum_{T\in\textrm{QYT}_{=k+1}(\lambda)} q^{ch(T)}.$$
Using the facts that $[k]\mapsto [k]\frac{1}{q^{k-1}}$ when substituting $1/q$ and that $\sum_{u\in\lambda} h(u) = n + n(\lambda) + n(\lambda')$, we get
$$W_k(P_\lambda,\omega;\frac{1}{q}) = \frac{q^{n(\lambda')}}{\prod_{u\in\lambda} [h(u)]} T_{n-k}^*(B_\lambda\times 1),$$
where $T_k^*(B_\lambda\times 1)$ gives a weight of $1/q$ to each circle instead of $q$. Multiplying this by $q^{{n\choose 2}}$ changes the circles to a $q^0$ weight and empty squares to a $q^1$ weight, which is identical to drawing circles downwards instead of upwards from crosses and giving circles a $q^1$ weight. Then by Proposition \ref{Blambda}, taking the complement of the board and reflecting vertically gives $B_{\lambda'}$ up to column permutation. By Theorem 7.13 of \cite{Dw98}, $T_k(B)$ is invariant on column permutations for Ferrers boards, so it follows that
$$T_{n-k}^*(B_\lambda\times 1) = \frac{T_k(B_{\lambda'})}{q^{{n\choose 2}}}.$$
This gives the following result, which clearly reduces to Theorem \ref{qytHit} when $q=1$. 
\begin{theorem}\label{qytCharge} 
Given a partition $\lambda$ and $0\leq k \leq n-1$,
$$\sum_{T\in \textrm{QYT}_{=k+1}(\lambda)} q^{ch(T)} = \frac{q^{nk+n(\lambda')-{n\choose 2}}}{\prod_{u\in\lambda} [h(u)]} T_k(B_{\lambda'}).$$
\end{theorem}
Summing over $k$ in this case also gives some sort of $q$-analogue of the hook-length formula, although it does not appear to immediately give a nice form.

\section{A summation formula}
We prove the following theorem in two ways, first with a $q$-hit number identity and then using $(P,\omega)$-partitions. This gives a relatively clean enumeration for quasi-Yamanouchi tableaux compared to the product formula of \cite{Wa16}, the downside being that it is not a positive summation. 
\begin{theorem}\label{qytSum}
Given a partition $\lambda$ and $0\leq k \leq n-1$
$$\textnormal{QYT}_{=k+1}(\lambda) = \sum_{m=0}^k{n+1\choose k-m}(-1)^{k-m} \textnormal{SSYT}_{m+1}(\lambda).$$
\end{theorem}

\begin{proof}[First proof.]

Setting $t=n$ in equation (24) in \cite{Ha98} gives 
$$\frac{T_{n-k}(B)}{\prod_{i=1}^n [d_i]!} = \sum_{m=0}^k {n+1 \brack k-m} (-1)^{k-m} q^{{k-m\choose 2}} \prod_{i=1}^n{m+H_i-D_i + d_i \brack d_i},$$
where $d_i =1$ for all $i$, $D_i = i$, and $H_i$ is the height of the $i$th column of $B$. By the way $B_\lambda$ is constructed, the sequence $H_i-D_i$ for $B_\lambda\times1$ becomes exactly the cell contents of $\lambda$, so setting $q=1$ gives
$$h_{n-k}(B_\lambda \times 1) =\sum_{m=0}^k {n+1\choose k-m}(-1)^{k-m}\prod_{i=1}^n (m+c_i+1).$$
Substituting this into Theorem \ref{qytHit} after applying Proposition \ref{Blambda} and comparing with the hook-content formula proves Theorem \ref{qytSum}.
\end{proof}

\begin{proof}[Second proof.]
When $\omega$ is a column-strict labeling on $P_\lambda$, $\mathcal{A}(P_\lambda,\omega;m)$ is the set of SSYT$_{m+1}(\lambda)$ with each entry decremented by one. Therefore, setting $q=1$ in $U_m(P_\lambda,\omega)$ gives $|\textrm{SSYT}_{m+1}(\lambda)|$. Proposition 8.4 in \cite{OSAP} says that
$$W_k(P,\omega) = \sum_{m=0}^k (-1)^m q^{{n\choose 2}} {n+1\brack m} U_{k-m}(P,\omega).$$
Then setting $q=1$, and reversing the order of summation proves Theorem \ref{qytSum}.
\end{proof}

\section{The polynomials $P_{n,k}$}

For any partition $\lambda$ with $|\lambda|=n$, we can actually express QYT$_{=k+1}(\lambda)$ in terms of certain symmetric functions $P_{n,k}$. We begin with Lemma 4 of \cite{Ha98}, where we set $q=1$, $t=n$, $d_i=1$, $e_i \in \{0,1\}$, $E_i$ the partial sums of the $e_i$, and $D_i = i$. We also recall as before that for $B=B_\lambda\times 1$, we have $H_i-D_i=c_i$, the cell contents of $\lambda$ in some order. After all of that, we get
$$h_{n-k}(B_\lambda\times1) = \sum_{e_1+\cdots+e_n = k} \prod_{i=1}^n {c_i+E_i+d_i-e_i\choose d_i-e_i}{i-1-c_i-E_i+e_i \choose e_i}.$$
Since exactly one of $d_i-e_i$ or $e_i$ are $1$ and the other is $0$, we get
$$h_{n-k}(B_\lambda\times1) = \sum_{e_1+\cdots+e_n = k} \prod_{i=1}^n (c_i+E_i+1)^{d_i-e_i}(i-c_i-E_i)^{e_i}.$$
This is the same as summing over weighted lattice paths with $n$ steps from $(0,0)$ to $(k,n-k)$. Let $E_i$ count the cumulative east steps and $N_i=i-E_i$ count the cumulative north steps. Then for each path, weight the $i$th step by $x_i+E_i+1$ if it is a north step and $N_i-x_i$ if it is an east step, and let the weight of a path be the product of the weights of its steps.

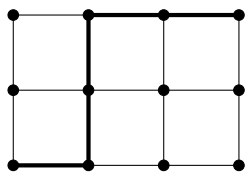
\begin{figure}[ht]
    \centering
    \begin{tikzpicture}
    \filldraw[black] (0,0) circle (2pt);
    \filldraw[black] (1,0) circle (2pt);
    \filldraw[black] (2,0) circle (2pt);
    \filldraw[black] (3,0) circle (2pt);
    \filldraw[black] (0,1) circle (2pt);
    \filldraw[black] (1,1) circle (2pt);
    \filldraw[black] (2,1) circle (2pt);
    \filldraw[black] (3,1) circle (2pt);
    \filldraw[black] (0,2) circle (2pt);
    \filldraw[black] (1,2) circle (2pt);
    \filldraw[black] (2,2) circle (2pt);
    \filldraw[black] (3,2) circle (2pt);
    \draw[black] (0,0) -- (0,2);
    \draw[black] (1,0) -- (1,2);
    \draw[black] (2,0) -- (2,2);
    \draw[black] (3,0) -- (3,2);
    \draw[black] (0,0) -- (3,0);
    \draw[black] (0,1) -- (3,1);
    \draw[black] (0,2) -- (3,2);
    \draw[black, ultra thick] (0,0) -- (1,0);
    \draw[black, ultra thick] (1,0) -- (1,1);
    \draw[black, ultra thick] (1,1) -- (1,2);
    \draw[black, ultra thick] (1,2) -- (3,2);
    \end{tikzpicture}
    \caption{A path with weight $(-x_1)(x_2+2)(x_3+2)(2-x_4)(2-x_5)$.}
    \label{fig:my_label}
\end{figure}

Let $P_{n,k}(x_1,\ldots,x_n)$ denote the sum of the weights of all such paths. This gives the following weighted lattice path interpretation for QYT enumeration. 
\begin{theorem}\label{qytLattice}
Given a partition $\lambda$ of $n$ with contents $c_1,\ldots,c_n$ and $0\leq k\leq n$,
$$\textnormal{QYT}_{=k+1}(\lambda) = \frac{P_{n,k}(c_1,\ldots,c_n)}{\prod_{u\in\lambda}h(u)}.$$
\end{theorem}
Such paths can be split recursively into ones that end on an east step and ones that end on a north step.
\begin{proposition}
The polynomials $P_{n,k}$ satisfy the relation
$$P_{n,k}(x_1,\ldots,x_n) = (x_n + k + 1) P_{n-1,k}(x_1,\ldots,x_{n-1}) + (n-k-x_n)P_{n-1,k-1}(x_1,\ldots,x_{n-1}).$$
\end{proposition}
We can use this to get a more concrete idea of what these polynomials look like. By their construction, it is not obvious that these polynomials are symmetric, but computing small cases seems to indicate they are.
\begin{align*}
P_{1,0}(x_1) =& e_1(x_1)+1 \\
P_{1,1}(x_1) =& e_1(x_1)\\
P_{2,0}(x_1,x_2) =& e_2(x_1,x_2) + e_1(x_1,x_2) + 1\\
P_{2,1}(x_1,x_2) =& -2e_2(x_1,x_2) -e_1(x_1,x_2) +1\\
P_{2,2}(x_1,x_2) =& e_2(x_1,x_2)\\
P_{3,0}(x_1,x_2,x_3) =& e_3(x_1,x_2,x_3) + e_2(x_1,x_2,x_3) + e_1(x_1,x_2,x_3) + 1\\
P_{3,1}(x_1,x_2,x_3) =& -3e_3(x_1,x_2,x_3) -2 e_2(x_1,x_2,x_3) +4\\
P_{3,2}(x_1,x_2,x_3) =& 3e_3(x_1,x_2,x_3)+e_2(x_1,x_2,x_3)-e_1(x_1,x_2,x_3)+1\\
P_{3,3}(x_1,x_2,x_3) =& e_3(x_1,x_2,x_3)
\end{align*}

Let $a(n,k,m)$ denote the coefficient of $e_m$ in $P_{n,k}$, and assume that $P_{i,k}$ is symmetric for $i<n$. Using the recursion, it is clear that the coefficient of the degree $m$ monomials containing $x_n$ in $P_{n,k}$ is $a(n-1,k,m-1) - a(n-1,k-1,m-1)$ and that the coefficient of the degree $m$ monomials not containing $x_n$ is $(k+1)a(n-1,k,m)+(n-k)a(n-1,k-1,m)$. Then to show that $P_{n,k}$ is symmetric, it is sufficient to show that $a(n-1,k,m-1) - a(n-1,k-1,m-1)=(k+1)a(n-1,k,m)+(n-k)a(n-1,k-1,m)$, which can be done with a straightforward induction argument. 
\begin{theorem}
Given a partition $\lambda$ of $n$ with contents $c_1,\ldots,c_n$ and $0\leq k \leq n$,
$$\textnormal{QYT}_{=k+1}(\lambda) = \frac{\sum_{m=0}^n a(n,k,m) e_m(c_1,\ldots,c_n)}{\prod_{u\in\lambda}h(u)}.$$
\end{theorem}

By the recursion and initial conditions, we have that $a(n,k,0)$ is the Eulerian number $A(n,k)$ and that for $1<m\leq n$, it is easy to generate these coefficients recursively using the relation $a(n,k,m) = a(n-1,k,m-1) - a(n-1,k-1,m-1)$. We also note that for a fixed value of $n-m$ with varying $n$ and $k$, this gives something close to a Pascal's triangle for the coefficients. Each term contributes its positive absolute value and its negative absolute value to the next line of the triangle, so summing over a line gives $0$ except when $m=0$. Therefore, summing over $P_{n,k}$ for all $0 \leq k \leq n$ leaves only the constant terms, and the hook-length formula is easily recovered.

\begin{figure}[ht]
    \centering
    \begin{tabular}{>{$n=}l<{$\hspace{12pt}}*{13}{c}}
3 &&&&&1&&4&&1&&&&\\
4 &&&&1&&3&&-3&&-1&&&\\
5 &&&1&&2&&-6&&2&&1&&\\
6 &&1&&1&&-8&&8&&-1&&-1&\\
\end{tabular}
    \caption{$a(n,k,m)$ for fixed $n-m=3$ and $0\leq k \leq n-1$ increasing along rows.}
    \label{figPascal}
\end{figure}

\section{Generating functions}
The Schur basis generating function for quasi-Yamanouchi tableaux is
\begin{equation}\label{qytGenFn}
\sum_{|\lambda| = n} \sum_{k=1}^n \textnormal{QYT}_{=k}(\lambda)t^{k-1}s_\lambda,
\end{equation}
which has a natural $q$-analogue
\begin{equation}\label{qytQGenFn}
\sum_{|\lambda| = n} \sum_{T\in \textnormal{QYT}(\lambda)} q^{\textnormal{maj}(T)} t^{\textnormal{des}(T)} s_\lambda.
\end{equation}
In this section, we present the fundamental quasisymmetric and monomial expansions of this $q$-analogue of the generating function. We note that the fundamental quasisymmetric expansion is an extension of Theorem 3.8 in \cite{FPSAC}.

\begin{theorem}
For $n \in \mathbb{N}$,
$$\sum_{\pi\in S_n} q^{\textnormal{maj}(\pi)}t^{\textnormal{des}(\pi)} F_{\textnormal{Des}(\pi^{-1})}(x) = \sum_{|\lambda|=n} \sum_{T\in \textnormal{QYT}(\lambda)} q^{\textnormal{maj}(T)}t^{\textnormal{des}(T)}s_\lambda.$$

\end{theorem}

\begin{proof}
Connect all $\pi \in S_n$ by colored edges corresponding to dual Knuth relations to get a graph $G$ and
 identify each permutation $\pi$ with its image $(P(\pi),Q(\pi))$ through RSK.
Dual Knuth relations do not change the descent set of a permutation, and the descent set of a permutation corresponds to the descent set of its recording tableau $Q(\pi)$. Therefore, all permutations in a connected component of $G$ have the same descent and major index statistics and map to the same recording tableau. 

On the other hand, RSK respects dual Knuth relations between permutations and their insertion tableaux, so the equivalence classes formed by dual Knuth relations guarantee that the insertion tableaux on a connected component range over exactly all $T\in \textnormal{SYT}(\lambda)$ for some $\lambda$. 
The descent set of an insertion tableau $P(\pi)$ is the same as the descent set of $\pi^{-1}$.
Then give each vertex of a connected component the weight $q^{\textnormal{maj}(\pi)}t^{\textnormal{des}(\pi)} F_{\textnormal{Des}(\pi^{-1})}(x)$ and apply Gessel's fundamental quasisymmetric expansion to show that each connected component has summed weight $q^{\textnormal{maj}(Q)}t^{\textnormal{des}(Q)} s_{sh(Q)}$, where $Q$ is the recording tableau shared by the connected component. RSK forms a bijection between $\pi \in S_n$ and pairs of SYT $(P,Q)$ of the same shape, so summing over all connected components of $G$, applying a counting argument, and using the correspondence between SYT and QYT completes the proof.
\end{proof}

For the monomial symmetric function expansion, we use multiset permutations. We can define descents and major index for multiset permutations in the same way as for permutations in $S_n$, and we write $S_{\lambda}$ for the set of multiset permutations of $\{1^{\lambda_1},2^{\lambda_2},\ldots\}$.

\begin{lemma} \label{genFnMono}
Given a partition $\lambda$ of $n$,
$$\sum_{\pi \in S_\lambda} q^{\textnormal{maj}(\pi)} t^{\textnormal{des}(\pi)} = \sum_{\nu \geq \lambda} K_{\nu\lambda} \sum_{T\in \textnormal{QYT}(\nu)} q^{\textnormal{maj}(T)} t^{\textnormal{des}(T)}.$$
\end{lemma}

\begin{proof}
RSK gives a bijection between multiset permutations $\pi \in S_\lambda$ and pairs of tableaux $(P,Q)$ of the same shape $\nu \geq \lambda$. In particular, $P$ is an SYT with descents in the same positions as $\pi$ and $Q$ has weight $\lambda$. Then since the descent set and major index are preserved, using the correspondence between SYT and QYT proves the claim. 
\end{proof}

\begin{theorem}
For $n \in \mathbb{N}$,
$$\sum_{|\lambda| = n} \sum_{\pi \in S_\lambda} q^{\textnormal{maj}(\pi)} t^{\textnormal{des}(\pi)} m_\lambda = \sum_{|\nu|=n}\sum_{T\in \textnormal{QYT}(\nu)} q^{\textnormal{maj}(T)}t^{\textnormal{des}(T)}s_\lambda.$$
\end{theorem}

\begin{proof}
We proceed by induction on the poset of partitions of $n$ under dominance order. The inductive claim is that the coefficient of $s_\lambda$ on the right hand side is the desired coefficient, and the inductive assumption is that the claim is true for all $\nu > \lambda$. As a base case, this clearly holds for $\lambda= (n)$ by computation. By the triangularity of the expansion of Schur functions into monomials, the coefficients of $m_\lambda$ on each side forces
$$\sum_{\pi \in S_\lambda} q^{\textnormal{maj}(\pi)} t^{\textnormal{des}(\pi)} = C_\lambda + \sum_{\nu > \lambda} K_{\nu\lambda} \sum_{T\in \textnormal{QYT}(\nu)} q^{\textnormal{maj}(T)} t^{\textnormal{des}(T)},$$
where $C_\lambda$ is the coefficient of $s_\lambda$ on the right hand side, and the second term comes from the expansion of each $s_\nu$, $\nu > \lambda$. Applying Lemma \ref{genFnMono} immediately shows that $C_\lambda = \sum_{T\in\textnormal{QYT}(\lambda)} q^{\textnormal{maj}(T)} t^{\textnormal{des}(T)}$. Continuing this induction downwards on the poset eventually proves the claim for all partitions of $n$. 
\end{proof}

\section{Applications}

We have already mentioned two instances of quasi-Yamanouchi tableaux proving to be a useful concept. The first was in Gessel's fundamental quasisymmetric expansion of Schur polynomials, due to Assaf and Searles \cite{AsSe17}. To make the statements in the introduction more precise, when the number of variables $x_1,x_2,\ldots$ is $k$, the standard Young tableaux that index the nonzero terms of the expansion are exactly those corresponding to quasi-Yamanouchi tableaux with maximum value at most $k$.
\begin{theorem}[Theorem 2.7, \cite{AsSe17}]
The Schur polynomial $s_\lambda(x_1,\ldots, x_k)$ is given by 
$$s_\lambda(x_1,\ldots,x_k) = \sum_{T\in \mathrm{QYT}_{\leq k}(\lambda)} F_{\textnormal{wt}(T)}(x_1,\ldots,x_k),$$
where all terms on the right hand side are nonzero.
\end{theorem}

The second instance was in the coefficients of Jack polynomials under a certain biniomial coefficient basis \cite{FPSAC}. More specifically, if we let $J_{\mu}^{(\alpha)}(x)$ denote the integral form, type A Jack polynomials and write $\Tilde{J}_{\mu}^{(\alpha)}(x) = \alpha^n J_{\mu}^{(1/\alpha)}(x)$, then we have the following theorem.

\begin{theorem}[Theorem 3.4, \cite{FPSAC}]
Let $\lambda$ be a partition of $n$ and $\lambda'$ be its conjugate. Then for the coefficient of $s_\lambda$ in $\tilde{J}_{(n)}^{(\alpha)}(X)$
$$\langle \tilde{J}_{(n)}^{(\alpha)}(x), s_\lambda\rangle = \sum_{k=0}^{n-1} a_k((n),\lambda) {\alpha + k \choose n},$$
we have $a_k((n),\lambda) =n!  \textnormal{QYT}_{=k+1}(\lambda')$.
\end{theorem}

On the representation theoretic side, QYT make another appearance with \emph{Foulkes characters}. First, associate a skew partition with a permutation as follows. For $\pi \in S_n$, let the signature $\sigma(\pi)$ of $\pi$ be a sequence of length $n-1$ of $+$s and $-$s so that $\sigma(\pi)_i = +$ if $i \not\in \textnormal{Des}(\pi)$ and $\sigma(\pi)_i = -$ if $i \in \textnormal{Des}(\pi)$. We can extend this definition to standard Young tableaux as well, using their definition of descent set.
To produce a ribbon shape given a signature $\sigma$, begin with a cell on a square grid, then at the $i$th step for $1 \leq i \leq n-1$, if $\sigma_i = -$, take a south step, and if $\sigma_i = +$, then take a west step, adding the cell at each step to the diagram. This gives a ribbon of length $n$, which we denote $R(\sigma)$. 
\begin{figure}[ht]
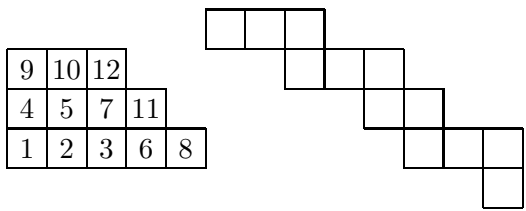

    \centering
$$\tableau{\\
9 & 10 & 12\\
4 & 5 & 7 & 11 \\
1 & 2 & 3 &6 & 8\\
}
\tableau{\ & \ & \ \\
&&\ & \ & \ \\
&&&& \ & \ \\
&&&&& \ & \ & \ \\
&&&&&&& \ \\}
$$
    \caption{The tableau on the left has signature $
\sigma = ++-++-+-++-$, which gives the ribbon $R(\sigma)$ on the right.}
    \label{figRuns}
\end{figure}

Kerber and Th\"urling \cite{KT84} obtain the decomposition of the skew representation $[R(\sigma)]$ using a ``cascade of diagrams" (see their paper for the full algorithm). 
They define the Foulkes character $\chi^{n,k}$ in to be
$$\chi^{n,k} = \sum_{\sigma} \chi^{R(\sigma)},$$
where the sum is over signatures $\sigma$ with exactly $k$ $+$s. 
If the order of nodes added in the cascade is recorded by placing an $i$ at the position of the $i$th node, their rules for producing the cascade from a particular signature create exactly all standard Young tableaux with the corresponding signature. Since the signature is essentially just the descent set, the correspondence between QYT and SYT means that the Foulkes character can be expressed as
$$\chi^{n,k} = \sum_{|\lambda|=n}\textnormal{QYT}_{=n-k}(\lambda) \chi^{\lambda}.$$
Define the \emph{Polya character} $\chi_n$

by
$$\chi_n(\pi) = m^{\#\textnormal{ of cycles of }\pi},$$
where $\pi \in S_n$. Kerber and Th\"urling showed that this has the decomposition
$$\chi_n = \sum_k {m+k\choose n} \chi^{n,k},$$
which from the QYT perspective gives the following result. 
\begin{proposition} The Polya character has the decomposition
$$\chi_n = \sum_k {m+k\choose n} \sum_{|\lambda|=n} \textnormal{QYT}_{=n-k}(\lambda) \chi^\lambda.$$
\end{proposition}

Now we get a nice surprise: replacing $\chi^\lambda$ by $s_\lambda$, letting $m=\alpha$, and multiplying by $n!$ gives precisely $\Tilde{J}_{(n)}^{(\alpha)}(x)$.

Our final example concerns the \emph{coinvariant algebra} $R_n$, which has been closely studied in both algebraic and geometric combinatorics. Recently for example, the Delta Conjecture of Haglund, Remmel, and Wilson \cite{HRW18} has received a great deal of attention. Haglund, Rhoades, and Shimozono \cite{HRS18} showed that a specialization of the combinatorial side of the Delta conjecture is the graded Frobenius image of a generalization of the coinvariant algebra. 

The coinvariant algebra is defined as
$$R_n = \frac{\mathbb{Q}[x_1,\ldots,x_n]}{\langle e_1(x_1,\ldots,x_n), \ldots, e_n(x_1,\ldots,x_n)\rangle}.$$
A common basis of $R_n$ is the \emph{Garsia-Stanton basis}, composed of monomials $gs_\pi (x_1,\ldots,x_n)$, $\pi \in S_n$, where
$$gs_\pi(x_1,\ldots,x_n) = \prod_{d \in \textnormal{Des}(\pi)} x_{\pi_1}\cdots x_{\pi_d}.$$
Lusztig and Stanley \cite{Sta79} gave the graded Frobenius image of $R_n$. 
$$\textnormal{grFrob}(R_n;q) = \sum_{|\lambda| = n} \sum_{T\in \textnormal{SYT}(\lambda)} q^{\textnormal{maj}(T)}s_{\lambda}(x).$$
Comparing with equation (\ref{qytQGenFn}), we can see that grFrob$(R_n;q)$ is equation (\ref{qytQGenFn}) at $t=1$. In this case, the degree of the monomial $gs_\pi$ captures the major index statistic. We can view the coinvariant algebra in another way that allows us to capture the descent statistic as well. Let $\mathcal{Y} = \{y_S \ |\  S\subseteq \{1,\ldots,n\}\}$ and 
$$\theta_i = \sum_{\substack{ S\subseteq\{1,\ldots,n\} \\ |S| = i}} y_S.$$ 
Garsia and Stanton \cite{GS84} proved that $R_n$ is isomorphic to
$$\mathcal{R}_n = \frac{\mathbb{Q}[\mathcal{Y}]}{\langle y_S\cdot y_T, \theta_1,\ldots,\theta_n \rangle},$$
where $y_S\cdot y_T$ is the product over $S$ and $T$ that are incomparable under inclusion ordering. They showed that we get a basis $\{y_\pi \ | \ \pi \in S_n\}$ defined by 
$$y_\pi = \prod_{d\in \textnormal{Des}(\pi)} y_{\{\pi_1,\ldots,\pi_d\}}.$$
Note that we can go between $y_\pi$ and $gs_\pi$ via the map that sends $y_{\{i_1,\ldots,i_k\}}$ to $x_{i_1}\cdots x_{i_k}$. If we assign to $y_S$ a $q$-degree of $|S|$ and a $t$-degree of $1$, then $y_\pi$ has a $q$-degree of maj$(\pi)$ and $t$-degree of des$(\pi)$. This produces a bigraded Frobenius image which is precisely equation (\ref{qytQGenFn}).

\begin{proposition} The bigraded Frobenius image of $R_n$ is
$$\textnormal{grFrob}(\mathcal{R}_n;q,t) = \sum_{|\lambda|=n} \sum_{T\in \textnormal{SYT}(\lambda)} q^{\textnormal{maj}(T)}t^{\textnormal{des}(T)}s_\lambda.$$
\end{proposition}

\section*{Acknowledgements}

I would like to thank Jim Haglund for the suggestion that sparked this work as well as all of his subsequent advice. I am also grateful to Brendon Rhoades for some helpful discussions, which were small in number but large in value, and to Sami Assaf for her comments and for originally introducing me to quasi-Yamanouchi tableaux.

\printbibliography

\end{document}